\documentclass{article}
\usepackage{amsmath, amsthm, amsfonts, amssymb}
\usepackage{hyperref}
\newtheorem{theorem}{Theorem}[section]
\newtheorem{lemma}[theorem]{Lemma}

\newtheorem{corollary}[theorem]{Corollary}
\theoremstyle{definition}
\newtheorem{definition}[theorem]{Definition}
\newtheorem{example}[theorem]{Example}

\theoremstyle{remark}
\newtheorem{remark}[theorem]{Remark}
\numberwithin{equation}{section}

\begin{document}
\title{\bf Topological Bicomplex Modules}
\date{\textbf{Romesh Kumar and Heera Saini}}
\vspace{0in}
\maketitle

$\textbf{Abstract.}$ In this paper, we develop topological modules over the ring of bicomplex numbers. We discuss bicomplex convexivity, hyperbolic-valued seminorms and  hyperbolic-valued Minkowski functionals  in bicomplex modules. We also study  the conditions under which  topological bicomplex modules and locally bicomplex convex modules become hyperbolic normable and     hyperbolic metrizable  respectively. \\\\
 $\textbf{Keywords.}$  Bicomplex modules, topological bicomplex modules,   bicomplex convexivity, hyperbolic-valued seminorms, hyperbolic-valued Minkowski functionals, locally bicomplex convex modules. 

\begin{section} {Introduction and Preliminaries}
 \renewcommand{\thefootnote}{\fnsymbol{footnote}}
\footnotetext{2010 {\it Mathematics Subject Classification}. 
30G35, 46A19, 46A22.}
\footnotetext{ } 
Bicomplex numbers are being studied for quite a long time. The book of G. B. Price \cite{KK} contains  the most comprehensive and extensive study of bicomplex numbers.
 Recently, a lot of work is being done on bicomplex functional analysis, see, e.g.,  \cite{CS}, \cite{H_6H_6},  \cite{ff}, \cite{G_1G_1}, \cite{GG}, \cite{LL}, \cite{Hahn}, \cite{Luna}, \cite{MM}. \cite{M_1M_1}, \cite{ZZ}, \cite{RR}, \cite{XX} and references therein.  A systematic study of functional analysis in this setting began with the monograph \cite{YY}. After publication of this monograph, some interesting papers have been published in this direction. Also see  \cite{KS} and\cite {RK} for recent work on bicomplex functional analysis. 

Topological vector spaces are one of the basic structures investigated in functional analysis.  The  bicomplex version of topological vector spaces was introduced in \cite{LL} and we are interested to develop a systematic  theory on these topological  structures. In this paper, we present some basic concepts and results on topological  modules over the ring of bicomplex numbers. Bicomplex convexivity plays a central role throughout the paper and has been discussed thoroughly in section 2. Some properties of hyperbolic-valued seminorms have been studied and their relationship with the sets which are bicomplex balanced,  bicomplex  convex and  bicomplex absorbing has been established in section 3.  Section 4 introduces the concept of hyperbolic-valued Minkowski functionals. These  functionals play an important role in the study of  locally bicomplex convex modules which have been discussed in section 5. For the study of topological vector spaces, we refer the reader to \cite{NB},  \cite{JJ}, \cite{DS}, \cite{LA}, \cite{NE}, \cite{HH}, \cite{rudin}  and \cite{YO} and references therein.\\

Now, we summarize some basic properties of bicomplex numbers. The set of bicomplex numbers $\mathbb{BC}$ is defined as 
$$
\mathbb{BC} =\left\{Z=w_{1}+jw_{2}\;|\;w_1 ,w_2 \in \mathbb{C}(i)\right\},
$$
where $i$ and $j$ are imaginary units such that $ ij=ji,  i^2=j^2=-1$ and $\mathbb{C}(i)$ is the set of complex numbers with the imaginary unit $i$. The set $\mathbb{BC}$ of bicomplex numbers  forms a  ring with the addition and multiplication defined as:
$$Z_1 +Z_2=(w_{1}+jw_{2})+(w_{3}+jw_{4})=(w_1 +w_3)+j(w_2 +w_4)\;,$$ $$Z_1 \cdot Z_2=(w_{1}+jw_{2})(w_{3}+jw_{4})=(w_1 w_3 -w_2 w_4)+j(w_2 w_3 +w_1 w_4).$$
Moreover,  $\mathbb{BC}$ is a module over itself.
The product of imaginary units $i$ and $j$ defines a hyperbolic unit $k$ such that $k^2=1$. The product of all units is commutative and satisfies $$ij=k, ~~ik=-j~~ \textmd{and}~~ jk=-i.$$
The set  of hyperbolic numbers $\mathbb{D}$ is defined as
$$\mathbb{D}=\left\{ \alpha = \beta_1 +k \beta_2 : \beta_1 , \beta_2 \in \mathbb{R} \right\}.$$  The set $\mathbb{D}$ of hyperbolic numbers  is a ring and a  module over itself.
Since the set $\mathbb{BC}$ has two imaginary units $i$ and $j$, two conjugations can be defined for bicomplex numbers and composing these two conjugations, we obtain a third one. We define these conjugations as follows:
\begin{enumerate}
\item[(i)] $Z^{\dagger_1}=\overline{w_1}+j\overline{w_2},$ 
\item[(ii)] $Z^{\dagger_2}={w_1}-j{w_2},$ 
\item[(iii)] $Z^{\dagger_3}= \overline{w_1}-j\overline{w_2},$
\end{enumerate}
where $\overline{w_1}, \overline{w_2}$ are   respectively the usual complex conjugates of ${w_1}, {w_2} \in \mathbb C(i)$.  For bicomplex numbers we have  the following three moduli:
\begin{enumerate}
\item[(i)] $|Z|^2_i=Z\;.\;Z^{\dagger_2}={w_1}^2+{w_2}^2 \in \mathbb{C}(i),$
\item[(ii)] $|Z|^2_j=Z\;.\;Z^{\dagger_1}= (|{w_1}|^2-|{w_2}|^2) + j(\;2\;\mathrm{Re}({w_1}\;\overline{{w_2}})) \in \mathbb{C}(j),$  
\item[(iii)] $|Z|^2_k=Z\;.\;Z^{\dagger_3}=(|{w_1}|^2+|{w_2}|^2) +k(\;-2\;\mathrm{Im}({w_1}\;\overline{{w_2}})) \in \mathbb{D}.$ 	 
\end{enumerate}
 A bicomplex number $Z={w_1}+j{w_2}$ is said to be invertible if and only if $$Z\;.\;Z^{\dagger_2}={w_1}^2+{w_2}^2 \not = 0$$ and its inverse is given by $$Z^{-1}=\frac{Z^{\dagger_2}}{|Z|^2_i}.$$ If $Z={w_1}+j{w_2} \not= 0$ is such that $Z\;.\;Z^{\dagger_2}={w_1}^2+{w_2}^2 =0$, then $Z$ is a zero divisor. The set of zero divisors $\mathcal {NC}$ of $\mathbb {BC}$ is, thus, given by
 $$\mathcal {NC} =\left\{Z\;|\; Z\neq 0,\; {w_1}^2+{w_2}^2=0 \right\},$$ and is called the null cone. Bicomplex algebra is considerably simplified by the introduction of two hyperbolic numbers $e_1$ and $e_2$ defined as
 $$e_1= \frac{1+k}{2} \;\; \textmd{and}\;\;\ e_2=\frac{1-k}{2}\;.$$
 The hyperbolic numbers $e_1$ and $e_2$ are zero divisors, which are linearly independent in the $\mathbb{C}$(i)- vector space $\mathbb{BC}$ and satisfy the following properties:
 $${e_1}^2=e_1,\;\;\; {e_2}^2=e_2,\;\; {e_1}^{\dagger_{3}}= e_1 , \;\;\;{e_2}^{\dagger_{3}}= e_2 ,\;\;\; e_1+e_2=1\;\;\; e_1\cdot e_2=0\;.$$ Any bicomplex number $Z={w_1}+j{w_2}$ can be uniquely written as 
\begin{equation}\label{eq1000} Z=e_1z_1+e_2z_2 \;,
\end{equation}  where $z_1=w_1 -iw_2 $ and $z_2=w_1 +iw_2 $ are elements of  $\mathbb{C}(i)$. Formula (\ref{eq1000}) is called the idempotent representation of a bicomplex number $Z$.
The sets $e_1\mathbb {BC}$ and $e_2\mathbb {BC}$ are ideals in the ring $\mathbb {BC}$ such that $$e_1\mathbb {BC} \; \cap \;e_2\mathbb {BC}= \left\{0\right\}$$ and 
\begin{equation}\label{eq1001} \mathbb {BC}= e_1\mathbb {BC}+e_2\mathbb {BC}\;.
\end{equation} Formula  (\ref{eq1001}) is called the idempotent decomposition of $\mathbb {BC}$. Writing a hyperbolic number $\alpha = \beta_1 +k \beta_2 $ in idempotent representation as 
$$ \alpha = e_1 \alpha_1 + e_2 \alpha_2, \;\;\; $$
where $\alpha_1 = \beta_1 + \beta_2 $ and $\alpha_2 = \beta_1 - \beta_2 $ are real numbers, we say that $\alpha$ is positive if $\alpha_1  \geq 0$ and $ \alpha_2 \geq 0 .$ Thus, the set of positive hyperbolic numbers $\mathbb{D}^{+}$ is given by
$$\mathbb{D}^{+} = \{ \alpha = e_1 \alpha_1 + e_2 \alpha_2 : \;\; \alpha_1 \geq 0,  \alpha_2 \geq 0\}.$$
For $\alpha, \gamma \in \mathbb{D},$ define a relation $\leq^{'}$ on $\mathbb{D}$ by $\alpha \; \leq^{'}\; \gamma$ whenever $\gamma - \alpha \in \mathbb{D}^{+}$. This relation is reflexive, anti-symmetric as well as transitive and hence defines a partial order on $\mathbb{D}.$ For further details on partial ordering on $\mathbb{D}$ one can refer \cite [Section 1.5] {YY}.\\
If $A \subset \mathbb{D}$ is $ \mathbb{D}$-bounded from  above, then the $\mathbb{D}$-supremum of $A$ is defined as $$ \sup_{\mathbb{D}}A = \sup A_1 e_1 +  \sup A_2 e_2,$$ where $A_1 = \{ a_1 \; : \; a_1 e_1 +a_2 e_2 \in A\}$ and $A_2 = \{ a_2 \; : \; a_1 e_1 +a_2 e_2 \in A\}$. Similarly, $\mathbb{D}$-infimum of  a  $ \mathbb{D}$-bounded below set $A$ is defined as
$$ \inf_{\mathbb{D}} A = \inf A_1 e_1 +  \inf A_2 e_2,$$ where $A_1$ and $A_2$ are as defined above.
\\The Euclidean norm $|Z|$ of a bicomplex number $Z=e_1z_1+e_2z_2$ is defined as
$$|Z|=\frac{1}{\sqrt{2}}\sqrt{|z_1|^2+|z_2|^2}\;.$$ One can easily check that $$|Z\cdot W| \leq \sqrt{2}|Z||W|,$$ for any $Z, W \in \mathbb{BC}.$
 The hyperbolic-valued or $\mathbb D$-valued norm $|Z|_k$ of a bicomplex number $Z=e_1z_1+e_2z_2$ is defined as  $$|Z|_k=e_1|z_1|+e_2|z_2|.$$ It is easy to see that $$|Z\cdot W|_k = |Z|_k \cdot |W|_k, $$ for any $Z, W \in \mathbb{BC}.$ The comparison of the Euclidean norm $|Z|$ and   hyperbolic-valued norm $|Z|_k$ of a bicomplex number $Z$ gives $$||Z|_k|=|Z|.$$
 The Euclidean norm and the hyperbolic-valued norm of bicomplex numbers have been discussed thoroughly  in \cite [Section 1.3, 1.5] {YY}.\\
A $\mathbb{B}\mathbb{C}$-module $X$ can be written as
\begin{equation}\label{eq1002} X=e_1X_1+e_2X_2\;,
\end{equation} where $X_1=e_1X$ and $X_2=e_2X$ are $\mathbb{C}(i)$-vector spaces as well as $\mathbb{B}\mathbb{C}$-modules(see, \cite{GG}, \cite{XX}). Formula (\ref{eq1002}) is called the idempotent decomposition of $X$. Thus, any $x$ in $X$ can be uniquely written as $x=e_1x_1+e_2x_2$ with $x_1\in X_1,\;x_2\in X_2$. Let $X$ be a $\mathbb{BC}$-module and $|| \cdot ||$ be a norm on $X$ considered as a vector space over $\mathbb{R}$. Then $|| \cdot ||$ is called a (real-valued) norm on the $\mathbb{BC}$-module $X$ if for any $x \in X$ and $z \in \mathbb{BC}$ $$||zx|| \leq \sqrt{2}\;|z|  \cdot ||x||.$$  Assume that $X_1$, $X_2$ are normed spaces with respective norms $\|.\|_1,\;\|.\|_2$. For any $x \in X$, set 
$$ \|x\|=\sqrt{\frac{\|x_1\|^2_1+\|x_2\|^2_2}{2}}\;\;.$$
 Then $\|.\|$ defines a real- valued norm on $\mathbb{BC}$-module $X$. This norm is called the Euclidean-type norm on $X$.\\
 Again, let $X$ be a $\mathbb{BC}$-module and $|| \cdot ||_{\mathbb{D}} : X \longrightarrow \mathbb{D}^{+}$ be a function such that for any $x, y \in X$ and $z \in \mathbb{BC},$ the following properties hold:
 \begin{enumerate}
\item[(i)] $|| x ||_{\mathbb{D}} = 0 \;\;\Leftrightarrow \;\;x=0.$
\item[(ii)] $|| zx ||_{\mathbb{D}} = |z|_{k}|| x ||_{\mathbb{D}}.$
\item[(iii)] 	 $|| x + y ||_{\mathbb{D}} \leq ^{'} || x ||_{\mathbb{D}} + || y ||_{\mathbb{D}}.$
\end{enumerate}
Then  $|| \cdot ||_{\mathbb{D}}$ is called a hyperbolic-valued ($\mathbb{D}$-valued) norm on $X$. If $X_1$, $X_2$ are normed spaces with respective norms $\|.\|_1,\;\|.\|_2$ then  $X$ can be endowed canonically with the hyperbolic-valued norm given by the formula  $$ \|x\|_{\mathbb D}=\|e_1x_1+e_2x_2\|_{\mathbb D}=e_1\|x\|_1+e_2\|x_2\|_2 \;.$$ The comparison of the real-valued norm $||x||$ and   hyperbolic-valued norm $||x||_{\mathbb{D}}$ of $x \in X$ gives $$|||x||_{\mathbb{D}}|=||x||.$$For more details on real-valued norm and hyperbolic-valued ($\mathbb D$-valued) norm see, \cite [Section 4.1, 4.2] {YY}.\\  For further details on bicomplex analysis, we refer the reader to \cite{YY}, \cite{G_1G_1}, \cite{GG},\cite{Hahn}, \cite{Luna},  \cite{MM}, \cite{M_1M_1}, \cite{ZZ}, \cite{KK},  \cite{RR}, \cite{XX} and references therein.  
\end{section}


\begin{section} {Topological Bicomplex Modules} 
Topological bicomplex modules have been introduced in  \cite [Section 2] {LL}. In this section, we introduce the concepts of balancedness, convexivity and absorbedness in bicomplex modules and discuss some of their properties.

\begin{definition}\label{def1}
Let $X$ be a $\mathbb{BC}$-module and $\tau$ be a Hausdorff topology on $X$ such that the operations
\begin{enumerate}
\item[(i)] $+\; : X \times X \longrightarrow X$  and
\item[(i)] $ \cdot\; : \mathbb{BC} \times X \longrightarrow X$
\end{enumerate}
are continuous. Then the pair $(X, \tau )$ is called a topological bicomplex module or topological $\mathbb{BC}$-module.
\end{definition}  
\begin{remark} A topological hyperbolic module can be defined in a similar way by just replacing $\mathbb{BC}$-module with $\mathbb{D}$-module in the above definition.
\end{remark}
\begin{remark}\label{rem2}Let $(X, \tau )$ be a topological $\mathbb{BC}$-module. Write $$X= e_1 X_1 + e_2 X_2, $$ where  $X_1=e_1X$ and $X_2=e_2X$  are $\mathbb{C}(i)$-vector spaces. Then 
  $\tau_l = \{ e_l G \; : \; G \in \tau \}$  
is a Hausdorff topology on $X_l$ for $l= 1, 2$. Moreover, the operations
\begin{enumerate}
\item[(i)] $+\; : X_l \times X_l \longrightarrow X_l$  and
\item[(i)] $ \cdot\; : \mathbb{BC} \times X_l \longrightarrow X_l$
\end{enumerate}
are continuous for $l= 1, 2$. Therefore, $(X_l, \tau_l)$ is a topological $\mathbb{C}(i)$-vector space for $l=1, 2.$
\end{remark}

\begin{example}\label{exm3}
Every $\mathbb{BC}$-module with $\mathbb{D}$-valued norm (or real-valued norm) is a topological $\mathbb{BC}$-module.
\end{example}

\begin{lemma} \label{lemma3a} $(\;\text{\cite [Lemma  2.1] {LL}}\;)$
For any $y \in X,$ the map $T_y : X \rightarrow X$ defined by $$T_y (x) = x + y , \;\; \textmd{for each}\; x \in X,$$ is a homeomorphism.
\end{lemma}
\begin{lemma} \label{lemma3b} $(\;\text{\cite [Lemma 2.2] {LL}}\;)$
For any  $\lambda  \in \mathbb{BC} \setminus \mathcal{NC}$,  the map $M_{\lambda} : X \rightarrow X$ defined by $$M_{\lambda} (x) = \lambda \cdot x, \;\; \textmd{for each}\; x \in X,$$ is a homeomorphism.
\end{lemma}

\begin{definition} \label{def12}
Let $B$ be a subset of a $\mathbb{BC}$-module $X$. Then $B$ is called a $\mathbb{BC}$-balanced set if for any $x \in B$ and $\lambda \in \mathbb{BC}$ with $|\lambda|_k \leq ' 1,\lambda x \in B$.
\end{definition}
 In other words, $\lambda B \subseteq B$ for any  $\lambda \in \mathbb{BC}, |\lambda|_k \leq ' 1$. It can be easily seen that  if  $B$ is a $\mathbb{BC}$-balanced set, then $0 \in B$. 
\begin{theorem} \label{thm12a}
Let $B$ be  a $\mathbb{BC}$-balanced subset of  a $\mathbb{BC}$-module $X$. Then
\begin{enumerate}
\item[(i)] $\lambda B = B$ whenever $\lambda \in \mathbb{BC} $ with $|\lambda|_k = 1.$
\item[(ii)]   $\lambda B = |\lambda|_k  B$ for each $   \lambda \in \mathbb{BC} \setminus \mathcal{NC}.$
\end{enumerate}
\end{theorem}
\begin{proof}
$(i) $ Let $\lambda \in \mathbb{BC}$ such that $ |\lambda|_k = 1$.  Since $B$ is  a $\mathbb{BC}$-balanced,  $\lambda B \subseteq B.$ Writing $\lambda = \lambda_1 e_1 + \lambda_2 e_2$, we have that $| \lambda_1 | =1$ and $|\lambda_2| =1$. Therefore, $$ \left| \frac{1}{\lambda} \right|_k = \frac{1}{|\lambda|_k} = {|\lambda|_k}^{-1} = | \lambda_1 |^{-1} e_1 + |\lambda_2|^{-1} e_1 =1. $$ It then follows that  $$  \frac{1}{\lambda} B \subseteq B.$$ Thus $B \subseteq \lambda B$.\\\\
$(ii)$   Let $\lambda \in \mathbb{BC} \setminus \mathcal{NC}.$  If $\lambda =0$, then clearly  $\lambda B = |\lambda|_k  B$.  So, we assume that $ \lambda \not= 0$. Writing $\lambda = \lambda_1 e_1 + \lambda_2 e_2$, we obtain $$ \frac{\lambda}{|\lambda|_k} = \frac{\lambda_1}{|\lambda_1 |} e_1 +  \frac{\lambda_2}{|\lambda_2 |} e_2. $$ Therefore $$ \left| \frac{\lambda}{|\lambda|_k} \right|_k =1. $$ By $(i)$, we have $$ \frac{\lambda}{|\lambda|_k} B =B.$$ Thus   $\lambda B = |\lambda|_k  B$. 
\end{proof}
\begin{theorem} \label{thm13}
Let $B$ be  a $\mathbb{BC}$-balanced subset of  a $\mathbb{BC}$-module $X$. Then
\begin{enumerate}
\item[(i)]  $e_1 B$ and  $e_2 B$ are balanced sets in $\mathbb{C}(i)$-vector spaces $e_1 X$ and $e_2 X$ respectively.
\item[(ii)]  $e_1 B \subset B$ and  $e_2 B \subset B$.
\end{enumerate}
\end{theorem}
\begin{proof} $(i)$ Let $x \in e_1 B$ and $a \in \mathbb{C}(i)$ such that $ |a| \leq 1.$ Then there exist $x' \in B$ and $a' \in \mathbb{BC}$ with $|a'|_k \leq ' 1$ such that $x = e_1 x'$ and $e_1 a=e_1 a'$. Since $B$ is $\mathbb{BC}$-balanced, $a' x' \in B$. Thus $ax = a  e_1 x' = e_1 a x'= e_1 a' x'  \in e_1 B$ showing that $e_1 B$ is a balanced set in $\mathbb{C}(i)$-vector space $e_1 X.$ Similarly, one can show that  $e_2 B$ is a balanced set in $\mathbb{C}(i)$-vector space $e_2 X.$\\\\
$(ii)$ Let $x \in e_1 B$. Then there is an $x' \in B$ such that $x = e_1 x'$. Since  $x' \in B$, by $\mathbb{BC}$-balancedness of $B$, $\lambda x' \in B$ for any  $\lambda \in \mathbb{BC},$ with $ |\lambda|_k \leq ' 1.$  In particular, taking $\lambda = e_1$, we get  $x = e_1 x' \in B$. Thus $e_1 B \subset B$. Similarly, it can be shown that $e_2 B \subset B$.
\end{proof}

\begin{definition} \label{def14}
Let $B$ be a subset of a $\mathbb{BC}$-module $X$. Then $B$ is called a $\mathbb{BC}$-convex set if  $x, y \in X$ and $\lambda \in \mathbb{D}^{+}$ satisfying $0 \leq ' \lambda \leq ' 1$ implies that $\lambda x + (1- \lambda) y \in B.$
\end{definition}
The definition of a  $\mathbb{BC}$-convex set is same as that of  $\mathbb{D}$-convex set except for the difference that a $\mathbb{BC}$-convex set is a subset of a  $\mathbb{BC}$-module and a $\mathbb{D}$-convex set is a subset of a  $\mathbb{D}$-module. The $\mathbb{D}$-convexivity for hyperbolic modules has been discussed in \cite [Section 9]{Hahn}. Proof of the next theorem is similar to the proof of \cite[ Proposition 18]{Hahn}, thus we omit it.
\begin{theorem} \label{thm15}
Let $B$ be  a $\mathbb{BC}$-convex subset of  a $\mathbb{BC}$-module $X$. Then 
\begin{enumerate}
\item[(i)]  $e_1 B$ and  $e_2 B$ are convex sets in $\mathbb{C}(i)$-vector spaces $e_1 X$ and $e_2 X$ respectively.
\item[(ii)] $e_1 B \subset B$ and  $e_2 B \subset B$ whenever $0 \in B$.
\end{enumerate}
\end{theorem}
The following lemma is easy to prove:
\begin{lemma}\label{lemma15a}
 In  a $\mathbb{BC}$-module $X$, if $\{ B_l \; :\; l \in \Delta \}$ is a collection of   $\mathbb{BC}$-convex sets, then $\cap_l B_l$ is  $\mathbb{BC}$-convex.
\end{lemma}
\begin{theorem} \label{thm16ab}
Let $B$ be a  $\mathbb{BC}$-convex subset of $\mathbb{BC}$-module  $X$. Then $B$ can be written as  $$B \; =\; e_1 B + e_2 B.$$
\end{theorem}
\begin{proof}  Let $x \in B$. Then $e_1 x \in e_1 B$ and $e_2 x \in e_2 B$. Therefore, $$x \; = \; (e_1 +e_2) x \; =\; e_1 x \; + e_2 x \; \in e_1 B + e_2 B$$ showing that $B \; \subset \; e_1 B + e_2 B$.  Now, let $x \in e_1 B$ and $ y \in e_2 B$. Then there exist $x',y' \in B$ such that $x= e_1 x'$ and $y=e_2 y'$. Since $B$ is  $\mathbb{BC}-$convex, we get $$ x+y \; = \; e_1 x' \; + \; e_2 y' \; =\;  \;   e_1 x' \; + \; (1-e_1 )  y' \in B.$$ Thus $\ e_1 B + e_2 B \; \subset B$. Hence, the proof.
\end{proof}
\begin{theorem} \label{thm16ac}
 Let $X$ be a $\mathbb{BC}$-module and $B \subset X$. If  $e_1 B$ and  $e_2 B$ are convex sets in $\mathbb{C}(i)$-vector spaces $e_1 X$ and $e_2 X$ respectively, then  $ e_1 B + e_2 B$ is  a $\mathbb{BC}$-convex subset of  $X$.
\end{theorem}
\begin{proof}  Let $x, y \in e_1 B + e_2 B$ and $ 0 \leq '  \lambda \leq ' 1$. Write $x = e_1 x_1 + e_2 x_2, \; y = e_1 y_1 + e_2 y_2$ and $\lambda = e_1  \lambda_1  + e_2 \lambda_2$, where $e_1x_1, e_1y_1 \in e_1 B,$ $e_2 x_2, e_2y_2 \in e_2 B$ and $ 0 \leq \lambda_1 , \lambda_2  \leq 1$. Since $e_1 B$ and  $e_2 B$ are convex in $\mathbb{C}(i)$-vector spaces $e_1 X$ and $e_2 X$ respectively, we have $$ e_1 \lambda_1   x_1  + e_1 (1- \lambda_1) y_1 \in  e_1 B$$ and  $$e_2  \lambda_2   x_2 + e_2 (1- \lambda_2) y_2 \in  e_2 B.$$ Thus,
\begin{eqnarray*}
\lambda x &+&  (1- \lambda y)\\
 &=& ( e_1  \lambda_1  + e_2 \lambda_2) ( e_1 x_1 + e_2 x_2 ) + (1 - (e_1  \lambda_1  +  e_2 \lambda_2)) (e_1 y_1 + e_2 y_2)\\
 &=& (   e_1  \lambda_1 x_1  + e_2 \lambda_2 x_2 ) + ( e_1(1- \lambda_1  )y_1 +e_2 (1- \lambda_2  )y_2)\\
 &=& (   e_1  \lambda_1 x_1  +   e_1(1- \lambda_1  )y_1 ) + (    e_2 \lambda_2 x_2+e_2 (1- \lambda_2  )y_2) \in e_1 B + e_2 B
\end{eqnarray*}showing that  $ e_1 B + e_2 B$ is $\mathbb{BC}$-convex.
\end{proof} If $e_1 B$ and  $e_2 B$ are convex sets in $\mathbb{C}(i)$-vector spaces $e_1 X$ and $e_2 X$ respectively, then $B =  e_1 B + e_2 B$ may not hold. Here is an example:
\begin{example}
Let $X = \mathbb{BC}$ and $B = \{ z =  e_1 z_1 + e_2 z_2 \; : \; z_1, z_2 \in \mathbb{C}(i),  |z_1| + |z_2| <2 \}.$ Then $e_1 B = \{e_1 z_1 \; : \; |z_1|< 2 \} $ and  $e_2 B = \{e_2 z_2 \; : \; |z_2|< 2 \} $ are convex sets in $\mathbb{C}(i)$-vector spaces $e_1 X$ and $e_2 X$ respectively. Observe that $e_1  \frac{3}{2} \in e_1 B$ and  $e_2 \frac{3}{2}  \in e_2 B$, but $ \frac{3}{2} = e_1 \frac{3}{2}   + e_2  \frac{3}{2}  \notin B$. Therefore, $B \not=  e_1 B + e_2 B$ and hence  $B$ is not $\mathbb{BC}$-convex.
\end{example}
\begin{theorem}\label{thm16af}
 Let $X$ be a topological $\mathbb{BC}$-module and $B \subset X$. Then the following statements hold:
\begin{enumerate}
\item[(i)] $ {(e_1 B)}^o = e_1{ B}^o $ \; and\; $ {(e_2 B)}^o = e_2{ B}^o$.
\item[(ii)] $\overline{ e_1 B} = e_1\overline{ B} $\; and\; $\overline{ e_2 B} = e_2 \overline{ B}. $
\end{enumerate}
\end{theorem}
\begin{proof}
$(i)$ Let $ x \in  {(e_1 B)}^o$. Then there exists an open neighbourhood $U \subset X$  such that $x \in e_1 U \subset e_1 B$. Let $x = e_1y$, for some $y \in U$. Clearly $y \in { B}^o$. Therefore $x= e_1 y \in e_1 { B}^o$, which proves that $ {(e_1 B)}^o \subset  e_1{ B}^o$.  Now, let $y \in { B}^o$ and $U \subset X$ be an open neighbourhood  such that $y \in  U \subset  B$. Then  $e_1 y \in e_1 U \subset e_1 B$. Since $U$ is open in $X$, $e_1 U$ is an open set in $e_1 X$. Therefore $e_1 y \in {(e_1 B)}^o$, showing that $e_1{ B}^o \subset {(e_1 B)}^o$. Thus $ {(e_1 B)}^o = e_1{ B}^o $.  Similarly, one has $ {(e_2 B)}^o = e_2{ B}^o$.\\\\
$(ii)$  Let $ x \in \overline{ e_1 B}$. Then there exists a net $\{ x_l\} \subset  e_1 B $ such that $x_l \rightarrow x$. Let $x_l = e_1 y_l$, where $y_l \in B$ and let  $y_l \rightarrow y$. Then $y \in \overline{ B}$. Therefore $x_l  = e_1 y_l \rightarrow e_1 y$. Since $e_1 X $ is  Hausdorff,  limits are unique, so $x= e_1 y \in e_1\overline{ B}$. This shows that $\overline{ e_1 B} \subset e_1\overline{ B}.$ Now, let $y \in \overline{ B}$. Then there exists a net $\{ y_l \} \subset  B $ such that $y_l \rightarrow y$. Therefore $\{ e_1 y_l \} \subset e_1  B $  such that  $ e_1 y_l   \rightarrow e_1 y$, showing that $ e_1 y \in \overline{ e_1 B}$. Thus $e_1\overline{ B} \subset \overline{ e_1 B}$, so we have $\overline{ e_1 B} = e_1\overline{ B}.$  Similarly, we have $\overline{ e_2 B} = e_2 \overline{ B}. $
\end{proof}

\begin{theorem}\label{thm16ag}
 Let  $B$ be a  $\mathbb{BC}$-convex set in a topological $\mathbb{BC}$-module $X$. Then the following statements hold:
\begin{enumerate}
\item[(i)] $B^o = e_1{ B}^o +  e_2{ B}^o$ and  $ \overline{ B} =  e_1\overline{ B} + e_2 \overline{ B}. $
\item[(ii)]${ B}^o$  and  $ \overline{ B}$ are  $\mathbb{BC}$-convex sets.
\end{enumerate}
\end{theorem}
\begin{proof}  $(i)$  Since $B$ is  $\mathbb{BC}$-convex, we can write $B$ as 
$$ B= e_1 B +  e_2 B.$$ Clearly, $ {B}^o \subset  e_1{ B}^o +  e_2{ B}^o$. Now, $e_1{ B}^o +  e_2{ B}^o$ is an open set in $X$ such that $ e_1{ B}^o +  e_2{ B}^o \subset e_1 B +  e_2 B = B$.  But ${B}^o$ is the largest open set contained in $B$. Therefore  $e_1{ B}^o +  e_2{ B}^o \subset {B}^o$.  Thus ${B}^o = e_1{ B}^o +  e_2{ B}^o$.  Again, trivially $   \overline{ B}  \subset  e_1\overline{ B} + e_2 \overline{ B}.$  From Theorem \ref{thm16af} and \cite[Theorem 1.13(b)]{rudin}, it follows that $ e_1\overline{ B} + e_2 \overline{ B} = \overline{e_1  B} +  \overline{e_2  B} \subset  \overline{e_1 B +  e_2 B} = \overline{B}.$ Thus, we have  $ \overline{ B} =  e_1\overline{ B} + e_2 \overline{ B}. $ \\\\
 $(ii)$ Since $B$ is  $\mathbb{BC}$-convex, $e_1 B$ and  $e_2 B$ are  convex sets in $\mathbb{C}(i)$-vector spaces $e_1 X$ and $e_2 X$ respectively. It then follows from  \cite [Theorem 1.13]{Hahn} that  $(e_1 B)^o$ and  $(e_2 B)^o$ are  convex sets  in $\mathbb{C}(i)$-vector spaces $e_1 X$ and $e_2 X$ respectively, so by Theorem \ref{thm16af}(i)  $e_1B^o$ and  $e_2 B^o$ are  convex sets  in $\mathbb{C}(i)$-vector spaces $e_1 X$ and $e_2 X$ respectively. From Theorem \ref{thm16ac} we see that   $e_1B^o + e_2 B^o$  is  $\mathbb{BC}$-convex. Finally, from $(i)$ it follows that $ B^o $ is a $\mathbb{BC}$-convex subset of  $X$.  Similarly, one can prove that $ \overline{ B}$ is  $\mathbb{BC}$-convex.
\end{proof}
\begin{theorem}\label{thm16ah}
 Let  $B$ be a $\mathbb{BC}$-balanced set in a topological $\mathbb{BC}$-module $X$. Then $ \overline{ B}$ is $\mathbb{BC}$-balanced and so is ${ B}^o$ if $0 \in B^o$.
\end{theorem}
\begin{proof} Let $\lambda \in \mathbb{BC}$ such that $ |\lambda|_k \leq ' 1$. If $\lambda =0$, then $\lambda  \overline{B} = \{ 0\} \subset\overline{B} $. If $ 0 < ' | \lambda|_k \leq ' 1$, then by Lemma \ref{lemma3b}, we have $\lambda  \overline{B} =  \overline{\lambda  B} \subset \overline{B}$. If $ \lambda = \lambda_1 e_1$ with $ 0< |\lambda_1| \leq 1$, then using respectively Theorem \ref{thm16af}(ii), \cite [ Theorem  2.1.2] {LA},  balancedness  of $e_1 B$  and Theorem \ref{thm13}(ii), we obtain    $\lambda  \overline{B} =  \lambda_1 e_1 \overline{B} =  \lambda_1 \overline{ e_1 B} =  \overline{  \lambda_1 e_1 B}  \subset \overline{  e_1 B} \subset  \overline{  B}.$ If $ \lambda = \lambda_2 e_2$ with $ 0< |\lambda_2| \leq 1$, then similarly as above, we get $\lambda  \overline{B}= \overline{B}$. Thus  $ \overline{ B}$ is $\mathbb{BC}$-balanced. Now, suppose that $0 \in B^o$ and $\lambda \in \mathbb{BC},\; |\lambda|_k \leq ' 1$. If $\lambda =0$, then $\lambda B^o = \{ 0\} \subset B^o $. If  $ 0 < ' | \lambda|_k \leq ' 1$, then using  Lemma \ref{lemma3b}, we get $\lambda B^o = {( \lambda B)}^o \subset B^o$.  If $ \lambda = \lambda_1 e_1$ with $ 0< |\lambda_1| \leq 1$, then using respectively Theorem \ref{thm16af}(i), \cite [ Theorem  2.1.2] {LA},  balancedness  of $e_1 B$  and Theorem \ref{thm13}(ii), we obtain $\lambda   B^o =  \lambda_1 e_1  B^o =  \lambda_1  (e_1 B)^o = ( \lambda_1  e_1 B)^o  \subset  (  e_1 B)^o \subset  B^o.$ Similarly, if $ \lambda = \lambda_2 e_2$ with $ 0< |\lambda_2| \leq 1$, then $\lambda B^o \subset B^o $. This proves that  ${ B}^o$ is  $\mathbb{BC}$-balanced whenever $0 \in B^o$
\end{proof}
\begin{definition}\label{def17}
Let $B$ be a subset of a $\mathbb{BC}$-module $X$. Then $B$ is called a $\mathbb{BC}$-absorbing set if for each $x \in X$, there exists $ \epsilon > ' 0$ such that $\lambda x \in B$ whenever  $ 0 \; \leq ' \;  \lambda \;  \leq ' \;  \epsilon$.
\end{definition}
It should be noted that a $\mathbb{BC}$-absorbing set always contains the origin.
\begin{theorem} \label{thm17a}
Let $B$ be  a $\mathbb{BC}$-absorbing set in  a $\mathbb{BC}$-module $X$. Then  $e_1 B$ and  $e_2 B$ are absorbing sets in $\mathbb{C}(i)$-vector spaces $e_1 X$ and $e_2 X$ respectively.
\end{theorem}
\begin{proof}
Let $x \in e_1 X.$ Then there exists $x' \in X$ such that $x = e_1 x'$. Since $B$ is $\mathbb{BC}$-absorbing, there exists $ \epsilon > ' 0$ such that $\lambda x \in B$ whenever  $ 0 \; \leq ' \;  \lambda \;  \leq ' \;  \epsilon $. Write   $\epsilon =  \epsilon_1 e_1 +  \epsilon_2 e_2$ and $ \lambda =  \lambda_1 e_1 +  \lambda_2 e_2.$ Then, we have  $\epsilon_1 > 0 $ and $$ e_1 ( \lambda x) = e_1 \lambda_1 x = \lambda_1 y \in e_1 B,$$ where $0 \leq  \lambda_1 \leq \epsilon_1$, proving that $e_1 B$ is an absorbing subset of $\mathbb{C}(i)$-vector spaces $e_1 X$. In a similar way, $e_2 B$ can be shown to be an absorbing subset of $\mathbb{C}(i)$-vector spaces $e_2 X$.
\end{proof}
\begin{remark} We have seen that if  $B$ is  a $\mathbb{BC}$-balanced set then $e_1 B \subset B$ and  $e_2 B \subset B$ and  if $B$ is  a $\mathbb{BC}$-convex set containing $0$ then also $e_1 B \subset B$ and  $e_2 B \subset B$. But it may happen that if $B$ is  a $\mathbb{BC}$-absorbing set and that it contains $0$, then neither  $e_1 B \subset B$ nor  $e_2 B \subset B$. We have the following example:
\end{remark}
\begin{example}
Let  $X = \mathbb{BC}$ and $B = \{ z \in \mathbb{BC} \; : \; |z|_k <' \frac{1}{2} \} \cup \{ 1 \}$. Then $B$ is a $\mathbb{BC}$-absorbing subset of  $\mathbb{BC}$. Now $1 \in B$, so we have $e_1 \in e_1 B$ and $e_2 \in e_2 B$. But neither  $e_1  \in B$ nor $e_2 \in B$  showing that $e_1 B  \nsubseteq B$ and $e_2  B  \nsubseteq B$.
\end{example}
The following result is essentially based on   \cite [Theorem 1.14] {rudin}.
\begin{theorem} 
Let  $(X, \tau ) $ be a topological  $\mathbb{BC}$-module. Then the following statements hold:
\begin{enumerate}
\item[(i)]Each neighbourhood of $0$ in $X$ is $\mathbb{BC}$-absorbing.
\item[(ii)] Each neighbourhood of $0$ in $X$ contains a $\mathbb{BC}$-balanced  neighbourhood of $0$.
\item[(iii)] Each  $\mathbb{BC}$-convex neighbourhood of $0$ in $X$ contains a $\mathbb{BC}$-convex  $\mathbb{BC}$-balanced  neighbourhood of $0$.
\end{enumerate}
\end{theorem}
\begin{proof} 
$(i)$  Let $U \subset X$ be a neighbourhood of $0$. Let $x \in X$. Since scalar multiplication is continuous and $\cdot (0,\; x) =0,$  there exists a neighbourhood $V$ of $x$ and $\epsilon > ' 0$ such that whenever $|\gamma|_k <' \epsilon$, we have $\gamma V \subset U$. In particular, for $\gamma$ satisfying $0 \leq ' \gamma \leq ' \epsilon /2 $, we have $\gamma x \in U$. This shows that $U$ is $\mathbb{BC}$-absorbing. \\\\
$(ii)$ Let $U \subset X$ be a neighbourhood of $0$. Since scalar multiplication is continuous  and $\cdot (0,\; 0) =0,$ there exists a neighbourhood $V$ of $0$ and $\epsilon > ' 0$ such that whenever $|\gamma|_k <' \epsilon$, we have $\gamma V \subset U$. Let $  M = \cup_{|\gamma|_k < ' \epsilon} \gamma V.$  Then $M$ is a neighbourhood of $0$ and $M \subset U$. To show that $M$ is $\mathbb{BC}$-balanced, let $x \in M$ and $|\lambda|_k \leq ' 1.$ Then $x = \gamma y$, for some $y \in V$. Since $|\lambda \gamma |_k = |\lambda |_k | \gamma |_k <' \epsilon$, it follows that $\lambda x = \lambda \gamma y \in M.$\\\\
$(iii)$  Let $U \subset X$ be a $\mathbb{BC}$-convex neighbourhood of $0$ and $  B = \cap_{|\gamma|_k = 1} \gamma U.$  Then  by $(ii)$, there is a $\mathbb{BC}$-balanced  neighbourhood say $M$ of $0$ such that $M \subset U$. For any $ \gamma \in \mathbb{BC}$ with $ |\gamma|_k = 1,$ we have by $\mathbb{BC}$-balancedness of $M$ that  $\gamma^{-1}M =M,$ so $M \subset  \gamma U$. Therefore $ M \subset B$, so $0 \in B^o \subset U$. We now show that $B^o$ is $\mathbb{BC}$-convex and $\mathbb{BC}$-balanced. Clearly,  $\gamma U$ is $\mathbb{BC}$-convex. It then follows from Lemma \ref{lemma15a} that $B$ is $\mathbb{BC}$-convex and so is $B^o$  by Theorem \ref{thm16ag}.  Since $\gamma U$ is $\mathbb{BC}$-convex set containing $0$, for any $\lambda \in \mathbb{D}^+ $, $ 0 \leq ' \lambda \leq '1$, we have  $\lambda  \gamma U \subset  \gamma U.$ Thus, for any $ \beta \in \mathbb{BC}$ with  $|\beta|_k = 1$  $$ \lambda \beta B = \cap_{|\gamma|_k = 1}\lambda \beta  \gamma U  = \cap_{|\gamma|_k = 1}\lambda  \gamma U  \subset  \cap_{|\gamma|_k = 1} \gamma U  = B.$$ This shows that $B$ is $\mathbb{BC}$-balanced. In view of Theorem \ref{thm16ah}   $B^o$ is also  $\mathbb{BC}$-balanced.
\end{proof}
\end{section}


\begin{section}{Hyperbolic-valued Seminorms}
The hyperbolic-valued seminorms in bicomplex modules have been studied in  \cite [Section 3] {Hahn}. In this section, we investigate  some properties of hyperbolic-valued seminorms in bicomplex modules as well as in topological  bicomplex modules.
\begin{definition}\label{def4}  \cite [Definition 2] {Hahn} Let $X$ be a $\mathbb{BC}$-module. Then a function $p \; : X \longrightarrow \mathbb{D} $ is said to be a hyperbolic-valued (or $\mathbb{D}$-valued) seminorm if for any $x, y \in X$ and $\lambda \in \mathbb{BC},$ the following properties hold:
 \begin{enumerate}
\item[(i)] $p(\lambda x)= |\lambda|_k p(x)$.
\item[(ii)] 	 $p( x + y ) \leq ^{'} p( x ) + p(y).$
\end{enumerate}
\end{definition} 

\begin{theorem}\label{thm5} Let $p$ be a $\mathbb{D}$-valued seminorm on a $\mathbb{BC}$-module $X$. Then for any $x, y \in X$, the following properties hold:
\begin{enumerate}
\item[(i)] $p(0) \; = \; 0.$
\item[(ii)] $|p(x)-p(y)|_{k} \leq ' p(x-y).$ 
\item[(iii)] $p(x) \geq ' 0$.
\item[(iv)] $ \{ x \; : \; p(x) =0 \}$ is a $\mathbb{BC}$-submodule of $X$.
\end{enumerate}
\end{theorem}
\begin{proof}
$(i)$   Using Definition \ref{def4}(i), we get 
\begin{eqnarray*} p(0)\; &=& \; p(0 \cdot x)\\
&=& \; |0|_k p( x)\\
&=& \; 0.
\end{eqnarray*}\\
$(ii)$ Using Definition \ref{def4}(ii), for any $x, y \in X$, we have
\begin{eqnarray*} p(x)\; & =& \; p(x-y+y)\\
&\leq '& \;  p( x-y) + p(y).
\end{eqnarray*}
Therefore
\begin{equation}\label{eq6} p(x) - p(y)\; \leq ' \; p(x-y).
\end{equation}
Again, \begin{eqnarray*} p(y)\; & =& \; p(y-x+x)\\
&\leq '& \;  p( y-x) + p(x)\\
&=& \; |-1|_k \; p(x-y) + p(x)\\
&=& \; p(x-y) + p(x),
\end{eqnarray*} which implies that
\begin{equation}\label{eq7} p(y) - p(x)\;  \leq ' \; p(x-y).
\end{equation}
Thus from equations (\ref{eq6}) and (\ref{eq7}), we obtain
$|p(x)-p(y)|_{k} \leq ' p(x-y).$\\\\
$(iii)$ By using $(ii)$ and then $(i)$, we obtain 
\begin{eqnarray*} p(x)\; & =& \; p(x-0)\\
&\geq '& \; | p(x) -p(0)|_k.\\
& = & |p(x)|_k \\
& \geq ' & 0.
\end{eqnarray*}\\
$(iv)$ Let  $ \lambda, \gamma \in \mathbb{BC}, $ and $x, y \in X$ such that $p(x)= p(y) =0$. Then 
\begin{eqnarray*}
  p( \lambda x + \gamma y)\;  &\leq ' &  \; |\lambda|_k p(x) + | \gamma |_k p(y)  \\
& = & 0.
\end{eqnarray*}
 Therefore, by $(iii)$, it follows that $  p( \lambda x + \gamma y) =0. $
\end{proof} 

\begin{remark} \label{rem8}Every $\mathbb{D}$-valued norm on a $\mathbb{BC}$-module is a  $\mathbb{D}$-valued seminorm. However, the converse is not true in general. Here is an example:
\end{remark}
\begin{example}\label{exm9}
Define a function $p \; : \; \mathbb{BC}  \longrightarrow \mathbb{D}$ by $$ p(Z)\; = \; |z_1|\; e_1, \;\;\;\;\; \textmd{for each}\; Z = z_1 e_1 + z_2 e_2 \in \mathbb{BC}.$$ Then $p$ is a $\mathbb{D}$-valued seminorm on $\mathbb{BC}$. In fact,\\
$(i)$ for any  $Z = z_1 e_1 + z_2 e_2, \; W= w_1 e_1 + w_2 e_2 \in \mathbb{BC},$ we have
\begin{eqnarray*} p(Z+W) &=& \;  |z_1 + w_1 | e_1\\
&\leq '&  ( |z_1| +| w_1 |) e_1  \\
&=&   |z_1| e_1+  | w_1 | e_1  \\
&=&  p(Z) + p(W); \;\textmd{and}
\end{eqnarray*}
$(ii)$ for any  $\lambda = \lambda_1 e_1 + \lambda_2 e_2, Z = z_1 e_1 + z_2 e_2 \in \mathbb{BC},$ we have
\begin{eqnarray*} p(\lambda Z) &=& \;  |\lambda_1 z_1| \; e_1\\
&=&   |\lambda_1|\; | z_1 | \; e_1 \\
&=&  (|\lambda_1| e_1 + |\lambda_2| e_2) \; |z_1| \;  e_1\\
&=& | \lambda|_k \; p(Z).
\end{eqnarray*}
Now $e_2 \not=0$, but $p(e_2) = 0$. Thus, $p$ is not a  $\mathbb{D}$-valued norm on $\mathbb{BC}$.
\end{example}
\begin{theorem} \label{thm10}
Let  $p$ be a  $\mathbb{D}$-valued seminorm on a topological $\mathbb{BC}$-module $X$. Denote the sets $\{ x \in X \; ; \; p(x) <' 1\}$ and $\{ x \in X \; ; \; p(x) \leq ' 1\}$ by $A$ and $C$ respectively. Then the following statements are equivalent:
\begin{enumerate}
\item[(i)] $p$ is continuous.
\item[(ii)]$A$ is open.
\item[(iii)] $ 0 \in A^{o}$.
\item[(iv)] $ 0 \in C^{o}$.
\item[(v)] $p$ is continuous at 0.
\item[(vi)] there exists a continuous  $\mathbb{D}$-valued seminorm $q$ on $X$ such that $p \leq 'q$.
\end{enumerate}
\end{theorem}
\begin{proof} We will prove this theorem in the following way:
$$ (i) \Rightarrow (ii) \Rightarrow (iii) \Rightarrow (iv) \Rightarrow (v) \Rightarrow (i) \;  \textmd{and} \; (i) \Leftrightarrow (vi) $$

$(i)  \Rightarrow  (ii):$ Let $ x_0 \in A$.  Choose $\epsilon > ' 0$ such that $p(x_0 ) + \epsilon < ' 1. $  Then the set $V= \{\lambda \geq ' 0 \; : \; | \lambda - p(x_0)  |_k < ' \epsilon \}$ is  a neighbourhood of $p(x_0)$ in $ \mathbb{D}^+$. Since $p$ is continuous at $x_0$, there exists a neighbourhood $U$ of $x_0$ such that $p(U) \subset V.$  Now, let $x \in U$. Then $p(x) \in V$. This implies that $| p(x) - p(x_0)  |_k < ' \epsilon $. Thus  $p(x) - p(x_0)  < ' \epsilon $. Therefore $p(x) < '  p(x_0) + \epsilon < ' 1,$ so $x \in A.$ Hence $U \subset A$. This shows that $A$ is open.

$(ii)  \Rightarrow (iii)$  and $ (iii)  \Rightarrow (iv)$ are obvious.

$(iv)  \Rightarrow (v) :$ Let $\epsilon >' 0$ and $\{x_l\}$ be a net in $X$ such that $x_l \rightarrow 0$ as $l \rightarrow \infty$. Then by (iv), there is an $l_0$ such that for $l \geq l_0$, we have $p(x_l) \leq ' \epsilon$. This shows that $p$ is continuous at 0.

$(v)  \Rightarrow (i) :$ Suppose $x_l \rightarrow x$. Then 
\begin{equation}\label{eq11}
|p(x_l)-p(x)|_{k} \leq ' p(x_l -x).
\end{equation} Since $x_l - x\rightarrow 0$ and $p$ is continuous at 0, it follows that $p(x_l - x)\rightarrow 0$. Equation (\ref{eq11}) implies that $p(x_l) \rightarrow p(x)$. Hence $p$ is continuous.

Again $ (i)  \Rightarrow (vi)$ is obvious.

$ (vi)  \Rightarrow (i) : $  Suppose $x_l \rightarrow x$. Then, since $x_l - x\rightarrow 0$ and $q$ is continuous, we obtain  $q(x_l - x)\rightarrow 0$. Now, $0 \leq ' p(x_l - x)  \leq ' q(x_l - x) $, so $p(x_l - x)\rightarrow 0$. Again, inequality (\ref{eq11}) yields  $p(x_l) \rightarrow p(x)$, proving that $p$ is continuous at $x$.
\end{proof}
\begin{theorem}\label{thm18}
Let $X$ be a $\mathbb{BC}$-module and $p$ be a   hyperbolic-valued seminorm on $X$. Then, $ \{ x \in X \; : \; p(x) < ' 1 \}$ and $ \{ x \in X \; : \; p(x)  \leq  ' 1 \}$  are $\mathbb{BC}$-convex , $\mathbb{BC}$-balanced and $\mathbb{BC}$-absorbing subsets of $X$.
\end{theorem}
\begin{proof} Let $ A\; =\; \{ x \in X \; : \; p(x) < ' 1 \}$ and $x, y \in A$.  Let $\lambda \in \mathbb{D}^{+}$ such that $0 \leq ' \lambda \leq ' 1$. Then 
\begin{eqnarray*}
p(\lambda x + (1 - \lambda ) y ) & \leq ' & \; p(\lambda x) + p((1 - \lambda ) y )  \\
&=& \; |  \lambda |_k  p(x) \; +\; |(1- \lambda) |_k  p(y) \\
&=&    \lambda  p(x) \; +\;( 1- \lambda)  p(y) \\
& < ' & \;   \lambda  \; +\; (1- \lambda) \\
&=& \; 1,
\end{eqnarray*}
proving that $A$ is  $\mathbb{BC}$-convex. Now, let $x \in A$ and $\lambda \in \mathbb{BC}$ with $|\lambda|_k \leq ' 1$. Then,
$ p(\lambda x )  =  \; |  \lambda |_k  p(x)
 \leq '  \;  p(x) 
 < '  \; 1. $
This shows that $A$ is  $\mathbb{BC}$-balanced. Let $x \in X$. Choose $\alpha > ' 0$ such that $\alpha > ' p(x).$  Set $\epsilon = 1/ \alpha$. Then, for any $\lambda$ satisfying  $ 0 \; \leq ' \;  \lambda \;  \leq ' \;  \epsilon $, we have $ p( \lambda x ) =    | \lambda |_k  p(x)  =  \lambda   p(x) \leq ' p(x) / \alpha  <' 1 $. Thus $A$ is $\mathbb{BC}$-absorbing. Similarly, it can be shown that the set $\{ x \in X \; : \; p(x)  \leq  ' 1 \}$ is $\mathbb{BC}$-convex , $\mathbb{BC}$-balanced and $\mathbb{BC}$-absorbing. 
\end{proof}
\end{section}


\begin{section}{Hyperbolic-valued Minkowski Functionals}
In this section, we define hyperbolic-valued Minkowski functionals in  bicomplex modules and it has been shown that a hyperbolic-valued Minkowski functional of a bicomplex balanced, bicomplex convex and bicomplex absorbing set turns out to be a   hyperbolic-valued seminorm. Hyperbolic-valued Minkowski functionals in  hyperbolic modules have been studied in \cite[Section 9]{Hahn}.
\begin{definition}\label{def19}
Let $B$ be a $\mathbb{BC}$-convex, $\mathbb{BC}$-absorbing subset of a $\mathbb{BC}$-module $X$. Then, the  mapping $q_B \; : \;  X \longrightarrow \mathbb{D}^{+}$ defined by $$ q_B (x) \; = \; \inf_{\mathbb{D}} \{ \alpha >' 0 \; : \; x \in \alpha B  \}, \;\;\; \textmd{for each} \; x \in X,$$ is called hyperbolic-valued  gauge or  hyperbolic-valued Minkowski functional of $B$. 
\end{definition}
\begin{theorem}\label{thm20}
Let  $B$ be a $\mathbb{BC}$-convex, $\mathbb{BC}$-balanced, $\mathbb{BC}$-absorbing subset of a $\mathbb{BC}$-module $X$. Then the $\mathbb{D}$-valued  gauge $q_B$ is a  $\mathbb{D}$-valued seminorm on $X$.
\end{theorem}
\begin{proof}
Let $x, y \in X$ such that $q_B (x) = \alpha$ and $q_B (y) = \gamma$. Then for any $\epsilon >'  0$, we have $x \in (\alpha + \epsilon )B $ and  $y \in (\gamma + \epsilon )B.$   Therefore, we can find $u, v \in B$ such that $x=  (\alpha + \epsilon ) u $ and $y = (\gamma + \epsilon ) v$. Observe that $ 0 < ' \frac{\alpha + \epsilon}{\alpha + \gamma+ 2 \epsilon} < ' 1$ and  $  0 < ' \frac{\gamma + \epsilon}{\alpha + \gamma+ 2 \epsilon} < ' 1$.  Therefore, by $\mathbb{BC}$-convexivity of $B$, we have $$ \left( \frac{\alpha + \epsilon}{\alpha + \gamma+ 2 \epsilon}\right)  u + \left( \frac{\gamma + \epsilon}{\alpha + \gamma+ 2 \epsilon}\right)  v \in B.$$ That is, $$ \frac{(\alpha + \epsilon) u + (\gamma+ \epsilon) v}{\alpha + \gamma + 2 \epsilon}  \in B.$$ This implies that $ (\alpha + \epsilon) u + (\gamma+ \epsilon) v \in (\alpha + \gamma +  2 \epsilon)  B$ which yields that $x \; + \; y \in (\alpha + \gamma +  2 \epsilon)  B.$ Letting $\epsilon \rightarrow 0$, we obtain $q_B (x+y)  \leq ' \alpha + \gamma =  q_B (x) + q_B (y).$
We now show that $q_B (\lambda x) \; = \; |\lambda|_k q_B $ for each $x \in X$ and $\lambda \in \mathbb{BC}$. Clearly, $ q_B (0) =0$. So, we assume that  $\lambda \not= 0$. Let $x \in X$. We first consider the case when $\lambda \in \mathbb{BC} \setminus \mathcal {NC}$. Since $B$ is $\mathbb{BC}$-balanced, by Theorem \ref{thm12a}, we have 
\begin{eqnarray*}
q_B (\lambda x) & =  & \;   \inf_{\mathbb{D}} \{ \alpha >' 0 \; : \;\lambda x \in \alpha B  \} \\
&=& \;  \inf_{\mathbb{D}} \left\{ \alpha >' 0 \; : \; x \in \alpha \left( \frac{1}{ \lambda} B \right)\right\} \\
&=& \;  \inf_{\mathbb{D}} \left\{ \alpha >' 0 \; : \; x \in \alpha \left( \frac{1}{ |\lambda|_k } B \right)\right\} \\
&=& \; |\lambda|_k \inf_{\mathbb{D}} \left\{\frac{\alpha}{ |\lambda|_k }  \; : \; x \in  \frac{\alpha}{ |\lambda|_k } B \right\} \\
&=& \; |\lambda|_k \; q_B ( x).
\end{eqnarray*}
Now, suppose $\lambda \in  \mathcal {NC}$. Then either $\lambda = \lambda_1 e_1$ or  $\lambda = \lambda_2 e_2$. We consider the case  $\lambda = \lambda_1 e_1$, the other one follows on similar lines.  Now, since $B$ is $\mathbb{BC}$-balanced set in $X$, it follows that $e_1 B$ is a balanced set in $\mathbb{C}(i)$-vector space $e_1 X$. Hence, 
\begin{eqnarray*}
q_B (\lambda x) & =  & \;   \inf_{\mathbb{D}} \{ \alpha >' 0 \; : \;\lambda x \in \alpha B  \} \\
& =  & \;   \inf_{\mathbb{D}} \{ \alpha  e_1 \; : \;\lambda_1 e_1  x \in \alpha e_1 B  \} \\
& =  & \;   \inf_{\mathbb{D}}  \left\{ \alpha  e_1 \; : \; e_1  x \in \alpha \left( \frac{1}{ \lambda_1} e_1 B \right) \right\} \\
& =  & \;   \inf_{\mathbb{D}}  \left\{ \alpha  e_1 \; : \; e_1  x \in \alpha \left( \frac{1}{| \lambda_1|} e_1 B \right) \right\} \\
& =  & \;   \inf_{\mathbb{D}}  \left\{ \alpha  e_1 \; : \; e_1  x \in  \frac{ \alpha}{| \lambda_1|} e_1 B   \right\} \\
& =  & \;  | \lambda_1| e_1 \inf_{\mathbb{D}}  \left\{ \frac{ \alpha}{| \lambda_1|}  \; : \;   x \in  \frac{ \alpha}{| \lambda_1|}  B   \right\} \\
& =  & \;  | \lambda_1| e_1 \; q_B (x)\\
& =  & \;  | \lambda|_k \; q_B (x).
\end{eqnarray*}
This completes the proof.
\end{proof}

\begin{definition}Let $(X, \tau )$ be a topological $\mathbb{BC}$-module. Then a subset $B \subset X$ is said to be bounded if for each neighbourhood $U$  of $0$, there exists $ \lambda >' 0 $ such that $B \subset \lambda U$.
\end{definition}
\begin{corollary}
Let  $B$ be a bounded $\mathbb{BC}$-convex, $\mathbb{BC}$-balanced,  $\mathbb{BC}$-absorbing subset of a topological $\mathbb{BC}$-module $(X, \tau)$. Then, $q_B$ is a  $\mathbb{D}$-valued norm on $X$.
\end{corollary}
The next result follows from Theorem \ref{thm18} and Theorem \ref{thm20}.
\begin{theorem}\label{thm21}
Let  $B$ be a $\mathbb{BC}$-convex, $\mathbb{BC}$-balanced,  $\mathbb{BC}$-absorbing subset of a $\mathbb{BC}$-module $X$ and $q_B$ be the  $\mathbb{D}$-valued  gauge on $B$. Then, both  the subsets 
$ \{x \in X \; | \; q_B (x) < ' 1  \} $ and $ \{  x \in X \; | \; q_B (x) \leq  ' 1  \} $ of $X$ are  $\mathbb{BC}$-convex, $\mathbb{BC}$-balanced and  $\mathbb{BC}$-absorbing.
\end{theorem}
\begin{theorem}\label{thm24}
Let $(X, \tau )$ be a topological $\mathbb{BC}$-module, $B$ be a $\mathbb{BC}$-convex, $\mathbb{BC}$-balanced,  $\mathbb{BC}$-absorbing subset of $X$ and $q_B$ be the  $\mathbb{D}$-valued  gauge on $B$. Let us denote $ \{x \in X \; : \; q_B (x) < ' 1  \} $ and $\{  x \in X \; : \; q_B (x) \leq  ' 1  \} $ by $A_{B}$ and $C_{B}$ respectively. Then, the following statements hold:
\begin{enumerate}
\item[(i)] $B^{o} \subset A_{B} \subset B \subset C_{B} \subset \overline{B}$.
\item[(ii)] If $B$  is open, then $  B =  A_{B} $.
\item[(iii)]  If $B$  is closed, then $  B = C_{B}$.
\item[(iv)] If  $q_B$ is continuous, then  $B^{o} = A_{B}$.
\end{enumerate} 
\end{theorem}
\begin{proof} $(i)$  Let $x_0 \in B^{o}.$ Then, there exists a neighbourhood $V$ of $x_0$ such that $x_0 \in V \subset B$. Since scalar multiplication is continuous, there exists $ \epsilon > ' 0$  and a neighbourhood $U$ of $x_0$ such that $\gamma x_0 \in V$ whenever $| \gamma - 1 |_k <' \epsilon$. It then follows that $x_0 \in \frac{1}{\gamma} B $ whenever $0 <' \gamma - 1 <'  \epsilon.$ Therefore, $q(x_0) \leq ' \frac{1}{\gamma} < ' 1,$ so $x_0 \in A_{B}$. This proves  that $B^{o} \subset  A_{B}$. Let $x_0 \in  A_{B}$.  Then, by Definition \ref{def19}, there exists $\alpha  $ satisfying $ 0< ' \alpha < ' 1$ such that $x_0 \in \alpha B$. Since $B$ is $\mathbb{BC}$-balanced, it follows that $\alpha B \subset B$. Thus, $x_0 \in B$. This shows that $ A_{B} \subset B$. Clearly $x_0 \in B$ implies that  $q_B (x_0) \leq  ' 1$. Thus  $B \subset C_{B}$.  Now, let $x_0 \in C_{B}$.  Then $q(x_0) \leq ' 1$ and we have the following cases: \\
Case $(i)$: If $q(x_0) <' 1,$ then $x_0 \in  A_{B} \subset  B \subset  \overline{B}$.\\
Case $(ii)$:  If $q(x_0)  = e_ 1 + r e_2,$ \; $0 \leq r <1$,  then take $a_n = (1 - \frac{1}{n})e_ 1 + e_2$, $n \in \mathbb{N}.$ We have \begin{eqnarray*}
q(a_n x_0) &=& |a_n|_k q(x_0)\\
&=&\left( (1 - \frac{1}{n})e_ 1 + e_2 \right) \left( e_ 1 + r e_2  \right)\\
&=&  (1 - \frac{1}{n})e_ 1 +  r  e_2 \\
&  < ' & 1.
\end{eqnarray*}
Therefore, $ a_n x_0  \in  A_{B} \subset B.$  Since $\lim_{n \rightarrow \infty} a_n x_0 = \lim_{n \rightarrow \infty} \left[ (1 - \frac{1}{n})e_ 1 + e_2 \right] x = (e_ 1 + e_2) x_0 =x_0, $ it follows that $x_0 \in  \overline{B}$.\\
Case $(iii)$:  If $q(x_0)  =  r e_ 1 + e_2,$ \; $0 \leq r <1$, we take $a_n =e_ 1 +  (1 - \frac{1}{n}) e_2$, $n \in \mathbb{N}.$ Then, similarly as in case $(ii)$ above, we obtain $x_0 \in  \overline{B}$.\\
Case $(iv)$:  If $q(x_0)  = 1$, take  $a_n = 1 - \frac{1}{n}$. Therefore 
\begin{eqnarray*}
q(a_n x_0) &=& |a_n|_k q(x_0)\\
&=& 1 - \frac{1}{n} \\
&  < ' & 1,
\end{eqnarray*}
which shows that $ a_n x_0  \in  A_{B} \subset B.$  Now, $\lim_{n \rightarrow \infty} a_n x_0 = \lim_{n \rightarrow \infty} (1 - \frac{1}{n}) x_0 =x_0, $ implies that $x_0 \in  \overline{B}$. Thus, in all the above cases $ C_{B}\subset \overline{B}$.\\\\
$(ii)$ If $B$ is open, then $B^{o} =B.$ Therefore, from $(i)$, we get $  B =  A_{B} $.\\\\
$(iii)$ If  $B$ is open, then $\overline{B} =B.$ Now, again using $(i)$, we obtain $  B = C_{B}$.\\\\
$(iv)$ Suppose, $q_B$ is continuous. Then by Theorem \ref{thm10},  $ A_{B} $ is an open set. Since $B^{o}$ is the largest open set containing $B$ and  $ A_{B} \subset B$, it follows that $ A_{B} \subset$ $B^{o}$. From $(i)$, we obtain $B^{o} = A_{B}$.
\end{proof}
\end{section}

\begin{section}{Locally Bicomplex Convex Modules}
In this section, we introduce the bicomplex version of locally convex  topological  spaces. Hyperbolic metrizable and hyperbolic normable locally bicomplex convex modules have also been studied.  Results in this section are analogous to the results on topological vectors spaces as given in \cite{JJ}, \cite{LA} and \cite{rudin}. We begin this section with the following definition.
\begin{definition}
Let $X$ be a  $\mathbb{BC}$-module and $\mathcal{P}$ be  a family of $\mathbb{D}$-valued seminorms on $X$. Then the family  $\mathcal{P}$ is said to be separated  if for each $x \not=0$, there exists  $p \in \mathcal{P}$ such that  $p(x) \not=0.$ 
\end{definition}
 We define a topology on a $\mathbb{BC}$-module $X$ determined by the family $\mathcal{P}$  of $\mathbb{D}$-valued seminorms on $X$   as follows:\\ 
For $x \in X, \epsilon >' 0$ and $p \in \mathcal{P}$, we set $$ U(x, \epsilon, p) = \{y \in X \; : \; p(y-x) <' \epsilon \};$$ and for $x \in X,  \epsilon >' 0$ and $ p_1, p_2,\;.\;.\;.\;, p_n \in \mathcal{P}$, set $$U(x, \epsilon , p_1, p_2,\;.\;.\;.\;, p_n ) = \{ y \in X \; : \; p_1 (y-x) <' \epsilon , p_2 (y-x) <' \epsilon ,\;.\;.\;.\;, p_n (y-x) <' \epsilon \}.$$  Let $\mathcal{U_{P}}(x) = \{ U(x, \epsilon , p_1, p_2,\;.\;.\;.\;, p_n )   : \epsilon>' 0, \;  p_1, p_2,\;.\;.\;.\;, p_n \in \mathcal{P} \; \textmd{and}\; n \in \mathbb{N}\}$. Then,  $\mathcal{U_{P}} = \{\mathcal{U_{P}}(x) \; : \; x \in X \} = \cup_{x \in X}\mathcal{U_{P}}(x) $ forms a base for a topology $\tau_{\mathcal{P}}$ on $X$, called the topology generated by the family $\mathcal{P}$.

\begin{theorem} Let  $ \mathcal{P} $ be a separated family of $\mathbb{D}$-valued seminorms on $X$. Then  $ (X, \tau_{ \mathcal{P}} )$ is a topological $\mathbb{BC}$-module.
\end{theorem}
\begin{proof} 
Let $x, y \in (X, \tau_{ \mathcal{P}} ) $ such that $x \not= y$. Then, there exists $p \in \mathcal{P}$ such that $p(x-y) \not= 0.$ Let $ p(x-y) = \epsilon = \epsilon_1 e_1 + \epsilon_2 e_2$. There arise the following three cases: \\
Case $(i)$: If $\epsilon_1 > 0 , \epsilon_2 > 0,$ take $\gamma = \epsilon/3$. Then $U(x, \gamma, p) = \{ z \; : \; p(x-z) < ' \gamma \} $ and  $V(y, \gamma, p) = \{ z \; : \; p(y-z) < ' \gamma\}$   are neighbourhoods of $x$ and $y$ in $ (X, \tau_{ \mathcal{P}} )$.  We shall show that $U \cap V \not= \phi.$ If possible, let $z_0 \in U \cap V$. Then $p(x- z_0) < ' \gamma$ and $p(y-z_0) < ' \gamma$. We have
\begin{eqnarray*}
p(x-y) & \leq ' &  p(x- z_0) + p(y- z_0)\\
& < '& \gamma + \gamma \\
&=& 2 \gamma,
\end{eqnarray*}
i.e., $\epsilon < ' 2 \gamma$, which is absurd. Thus, $U$ and $V$ are disjoint neighbourhoods of $x$ and $y$ in $ (X, \tau_{ \mathcal{P}} )$.\\
Case $(ii)$: If $\epsilon_1 >  0 , \epsilon_2 = 0,$ take $ \gamma = \frac{\epsilon_1}{3} e_1 + r e_2$, where $r >0$.  Then $U(x, \gamma, p) = \{ z \; : \; p(x-z) < ' \gamma \}$ and  $V(y, \gamma, p) = \{ z \; : \; p(y-z) < ' \gamma\} $  are neighbourhoods of $x$ and $y$ in $ (X, \tau_{ \mathcal{P}} )$. Now, let $z_0 \in U \cap V$. Then $p(x- z_0) < ' \gamma$ and $p(y-z_0) < ' \gamma$. Therefore
\begin{eqnarray*}
p(x-y) & \leq ' &  p(x- z_0) + p(y- z_0)\\
& < '& \gamma + \gamma \\
&=& 2 \gamma,
\end{eqnarray*}
i.e., $\epsilon_1 e_1 < ' 2 \frac{\epsilon_1}{3} e_1 + 2 r e_2$, which is not possible. Thus, $U$ and $V$ are disjoint neighbourhoods of $x$ and $y$ in $ (X, \tau_{ \mathcal{P}} )$.\\
Case $(iii)$: If $\epsilon_1 = 0 , \epsilon_2 > 0,$ take $ \gamma = s  e_1 + \frac{\epsilon_2}{3}e_2$, where $s > 0$. Then as in case $(ii)$,  one easily checks that $U(x, \gamma, p) = \{ z \; : \; p(x-z) < ' \gamma \}$ and  $V(y, \gamma, p) = \{ z \; : \; p(y-z) < ' \gamma\} $  are disjoint neighbourhoods of $x$ and $y$. This shows that  $ (X, \tau_{ \mathcal{P}} )$ is a Hausdorff space.
 We now show that  $+\; : X \times X \longrightarrow X$ is continuous. For this, let $(x_0 , \; y_0) \in X \times X$ and $ U(x_0 + y_0, \epsilon , p_1, p_2,\;.\;.\;.\;p_n )$ be a  neighbourhood of  $x_0 +  \; y_0 $ in  $ (X, \tau_{ \mathcal{P}} )$. Then $U(x_0, \epsilon /2 , p_1, p_2,\;.\;.\;.\;p_n ) \times U( y_0, \epsilon /2 , p_1, p_2,\;.\;.\;.\;p_n )$ is a  neighbourhood of  $(x_0 , \; y_0) $ in  the product topology of $X \times X$. Let $(x, y) \in U(x_0, \epsilon /2 , p_1, p_2,\;.\;.\;.\;p_n ) \times U( y_0, \epsilon /2 , p_1, p_2,\;.\;.\;.\;p_n )$. Then $p_i (x - x_0) <' \epsilon /2$ and $p_i (y - y_0) <' \epsilon /2$ for each $i= 1,2,\;.\;.\;.\;, n.$ Therefore, for $i= 1,2,\;.\;.\;.\;, n,$ we have 
\begin{eqnarray*}
p_i [(x+y) -(x_0 +y_0)] &=& p_i [(x- x_0) + (y -y_0)]\\
 & \leq '&  p_i (x- x_0) + p_i (y -y_0) \\
 & < '&  \epsilon /2 +  \epsilon /2 \\
 & = & \epsilon.
\end{eqnarray*}
That is, $+ (x, \; y) \in U(x_0 + y_0, \epsilon , p_1, p_2,\;.\;.\;.\;p_n )$, proving the continuity of  $+\; : X \times X \longrightarrow X$ at $( x_0,  y_0)$. Since  $( x_0,  y_0)$ is arbitrary, it follows that $+\; : X \times X \longrightarrow X$ is continuous. Now, to show that  $ \cdot\; : \mathbb{BC} \times X \longrightarrow X$ is continuous, let $(\lambda_0, x_0) \in  \mathbb{BC} \times X$ and $ U(\lambda_0 x_0, \epsilon , p_1, p_2,\;.\;.\;.\;p_n )$ be a  neighbourhood of  $\lambda_0 x_0  $ in  $ (X, \tau_{ \mathcal{P}} )$. Choose $\delta > ' 0$ such that $\delta^{2} +(|\lambda_0|_k + p_i(x_0)) \delta < ' \epsilon$ for each $i= 1,2,\;.\;.\;.\;, n.$  Let $U_{\delta} = \{ \lambda \in \mathbb{BC} \; : \; |\lambda  - \lambda_0|_k <' \delta \}$.  Then $ U_{\delta}  \times U(x_0, \delta, p_1, p_2,\;.\;.\;.\;p_n )$ is a  neighbourhood of $(\lambda_0, x_0) $ in   $ \mathbb{BC} \times X$. Let $(\lambda, x) \in  U_{\delta}  \times U(x_0, \delta, p_1, p_2,\;.\;.\;.\;p_n )$. We then have  for $i= 1,2,\;.\;.\;.\;, n,$
\begin{eqnarray*}
p_i (\lambda x - \lambda_0 x_0 ) &=&  p_i [ (\lambda - \lambda_0) (x - x_0 ) + \lambda_0 (x -x_0 ) + x_0 ( \lambda - \lambda_0) ] \\
& \leq ' &  p_i [ (\lambda - \lambda_0) (x - x_0 )  ]  +  p_i [  \lambda_0 (x -x_0 ) ]   +  p_i [  x_0 ( \lambda - \lambda_0)  ] \\
 &=&  |(\lambda - \lambda_0|_k p_i  (x - x_0 ) + |\lambda_0|_k p_i  (x - x_0 )  +   |(\lambda - \lambda_0|_k p_i (x_0) \\
& < ' & \delta^{2} +(|\lambda_0|_k + p_i(x_0)) \delta \\
& < ' & \epsilon.
\end{eqnarray*}
Therefore, $\cdot (\lambda, x) \in  U(\lambda_0 x_0, \epsilon , p_1, p_2,\;.\;.\;.\;p_n )$ showing that  $ \cdot\; : \mathbb{BC} \times X \longrightarrow X$ is continuous at $(\lambda_0, x_0)$. Since $(\lambda_0, x_0)$  is  arbitrary, it follows that $ \cdot\; : \mathbb{BC} \times X \longrightarrow X$ is continuous. Hence,  $ (X, \tau_{ \mathcal{P}} )$ is a topological $\mathbb{BC}$-module.
\end{proof}
 Proof of the following lemma is on the  similar lines as in  \cite [Lemma 2.5.1] {LA}, so we omit the proof.
\begin{lemma}\label{th28}
Let $X $ is a topological $\mathbb{BC}$-module and $\mathcal{P} = \{ p_n \}_{n \in \mathbb{N}}$ be  a family of  $\mathbb{D}$-valued seminorms on $X$.  For each $m \in \mathbb{N}$, define $q_m \; : \; X \rightarrow \mathbb{D}$ by  $$ q_m(x)  = \sup \{ p_1 (x),  p_2 (x), \; . \; . \; .\; ,  p_m (x) \}, \;\;\; \textmd{for each} \; x \in X.$$ Then,  $\mathcal{Q} = \{ q_m \}_{m \in \mathbb{N}}$  is a  family of $\mathbb{D}$-valued seminorms on $X$ such that the following hold:
\begin{enumerate}
\item[(i)]  $\mathcal{Q}$ is separated if $\mathcal{P}$ is so.
\item[(ii)] $q_m \leq ' q_{m+1},$ for each $m \in \mathbb{N}$.
\item[(iii)] $(X, \tau_{\mathcal{P}})$ and $(X, \tau_{\mathcal{Q}})$ are topologically isomorphic.
\end{enumerate}
\end{lemma}

\begin{definition}\label{def29}
A topological $\mathbb{BC}$-module $(X, \tau)$ is said to be locally bicomplex convex (or $\mathbb{BC}$-convex)  module  if it has a neighbourhood base at $0$ of  $\mathbb{BC}$-convex sets.
\end{definition}
Proof of the following theorem is on the  similar lines as in  \cite [Theorem 2.3.1] {LA}. For the sake of completeness, we give the proof.
\begin{theorem}\label{th30}
 A topological $\mathbb{BC}$-module $(X, \tau)$ is a locally $\mathbb{BC}$-convex module if and only if its topology is generated by   a separated  family $\mathcal{P}$ of  $\mathbb{D}$-valued seminorms on $X$.
\end{theorem}
\begin{proof} Suppose $(X, \tau)$ is a locally $\mathbb{BC}$-convex module. Let $\mathcal{B}$ be a neighbourhood base  at $0$ of  $\mathbb{BC}$-convex sets. Assume that each $B \in \mathcal{B}$ is $\mathbb{BC}$-balanced and $\mathbb{BC}$-absorbing and let $q_B$ be the $\mathbb{D}$-valued Minkowski functional of $B$. Then $\mathcal{P}= \{q_B \}$ is a family of $\mathbb{D}$-valued seminorms on $X$ such that for each $B$, $$ \{x \in X : q_B (x) <' 1 \} \subset B \subset \{x \in X : q_B (x) \leq ' 1   \}. $$ Therefore the topology on $X$ is generated by  $\mathcal{P}$. Let $0  \not= x \in X$.  Then, there exists $B \in \mathcal{B}$ such that $x \notin B$. For this $B$, we have $q_B(x) \geq ' 1$. Thus, $\mathcal{P}= \{q_B \}$ is a separated family. 

Conversely, suppose $\mathcal{P}$ is a separated family of  $\mathbb{D}$-valued seminorms on $X$ that generates the topology on $X$. Let $\{ U(0, \epsilon , p_1, p_2,\;.\;.\;.\;, p_n )   : \epsilon>' 0,   p_1, p_2,\;.\;.\;.\;, p_n \in \mathcal{P} \}$ be a neighbourhood base at $0$. Let $0 \leq' \lambda \leq' 1$ and $x, y \in U(0, \epsilon , p_1, p_2,\;.\;.\;.\;, p_n )$. Then, for each $m = 1, 2, \; .\; .\;. \; , n$, 
\begin{eqnarray*}
p_m (\lambda x + (1-\lambda) y) & \leq ' &  p_m (\lambda x ) + p_m( (1-\lambda) y) \\
&=&  \lambda  p_m ( x ) +  (1-\lambda)  p_m( y)\\
& < ' &  \lambda \epsilon + (1-\lambda) \epsilon \\
& =& \epsilon .
\end{eqnarray*}
This shows that $ U(0, \epsilon , p_1, p_2,\;.\;.\;.\;, p_n )$ is $\mathbb{BC}$-convex.
\end{proof}

\begin{definition}
Let   $d \; : \; X \times X \rightarrow  \mathbb{D}$ be a function such that for any $x, y, z \in X$, the following properties hold:
 \begin{enumerate}
\item[(i)] $d(x, y) \geq ' 0$  and  $d(x, y) = 0$ if and only if $x =y$,
\item[(ii)]  $ d(x, y) = d(y, x$),
\item[(iii)] $d(x, z) \leq' d(x, y) + d(y, z)$.
\end{enumerate}
Then $d$ is called a hyperbolic-valued (or  $\mathbb{D}$-valued) metric on $X$ and the pair $(X, d)$ is called  a   hyperbolic metric (or $\mathbb{D}$-metric) space.
\end{definition} The following result is easy to prove
\begin{lemma}
Every  $\mathbb{D}$-metric space is first countable.
\end{lemma}
\begin{definition}
A  topological $\mathbb{BC}$-module $X$ is said to be hyperbolic metrizable (or  $\mathbb{D}$-metrizable)  if the topology on $X$ is generated by a $\mathbb{D}$-valued metric on $X$.
\end{definition}

\begin{definition}
A  topological $\mathbb{BC}$-module $X$ is said to be  hyperbolic normable (or $\mathbb{D}$-normable) if the topology on $X$ is generated by a $\mathbb{D}$-valued norm on $X$.
\end{definition}
 Proof of the following lemma is on the  similar lines as in the case of topological vector spaces ( see, e.g.,  \cite [Theorem 2.5.1] {LA}).
\begin{lemma}\label{lemma40}
 Let $ \mathcal{P} =\{p_n \}$ be a  countable separated family of $\mathbb{D}$-valued seminorms on a  topological $\mathbb{BC}$-module $(X, \tau)$  such that $p_n \leq ' p_{n+1}$ for each $n \in \mathbb{N}$.
Define a function  $d : X \times X  \rightarrow  \mathbb{D}$  by
$$ d(x, y) =\sum_{n=1}^{\infty} 2^{-n} \frac{p_n(x-y )}{1 + p_n(x-y )}, \; \textmd{for each}\;   x, y \in X.$$
Then, $d$ is a translation invariant $\mathbb{D}$-valued metric on $X$ and the topology on $X$ generated by $d$ is the topology generated by the family   $\mathcal{P}$.
\end{lemma}
\begin{theorem} A locally $\mathbb{BC}$-convex module $(X, \tau)$ is $\mathbb{D}$-metrizable if and only if its topology is generated by   a countable  separated  family $\mathcal{P}$ of  $\mathbb{D}$-valued seminorms on $X$.
\end{theorem}
\begin{proof} Suppose  $X$ is $\mathbb{D}$-metrizable. Then $X$ has a countable  neighbourhood base $\mathcal{B}= \{ B_n \; : \; n \in \mathbb{N} \}$  at $0$ of  $\mathbb{BC}$-convex sets. Assume that  each $B_n$ is $\mathbb{BC}$-balanced and $\mathbb{BC}$-absorbing  and $B_{n+1} \subset B_n$  for each $n \in \mathbb{N}$. Let $q_n$ be the $\mathbb{D}$-valued Minkowski functional of $B_n$. Then as in the  proof of Theorem \ref{th30},  $\mathcal{Q}= \{q_n \}$ is a separated  family $\mathcal{P}$ of  $\mathbb{D}$-valued seminorms on $X$ that generates the topology on $X$.

Conversely,  suppose $\mathcal{P}= \{p_n \}$ is a countable separated family of  $\mathbb{D}$-valued seminorms on $X$ that generates the topology on $X$. By Lemma \ref{th28}, we may assume that the family $\{p_n \}$  of  $\mathbb{D}$-valued seminorms is increasing. The rest of the proof follows from Lemma \ref{lemma40}.
\end{proof}

\begin{theorem} A topological $\mathbb{BC}$-module $(X, \tau)$ is $\mathbb{D}$-normable if and only if it contains a bounded $\mathbb{BC}$-convex neighbourhood of $0.$ 
\end{theorem}
\begin{proof} Suppose, $(X, \tau)$ is $\mathbb{D}$-normable. Then, clearly,  the set $B = \{x \in X \; :\; ||x||_{\mathbb{D}}< ' 1 \}$ is a bounded $\mathbb{BC}$-convex neighbourhood of $0.$ 

Conversely,  suppose $B$ is a bounded $\mathbb{BC}$-convex neighbourhood of $0.$ We may assume that $B$ is $\mathbb{BC}$-balanced. For each $x \in X$, define
\begin{equation}\label{eq41}
||x||_{\mathbb{D}}= q_B (x),
\end{equation}
where $q_B$ is the $\mathbb{D}$-valued Minkowski functional of $B$.  Then, boundedness of $B$ implies that  (\ref{eq41}) is a  $\mathbb{D}$-valued norm on $X$. Since $B$ is open, $B = \{x \in X \; :\; ||x||_{\mathbb{D}}< ' 1 \}$. We show that $\{ \lambda B \; : \; \lambda >' 0 \}$ is a neighbourhood base at $0$. For this, let $U$ be a neighbourhood of $0$. Since $B$ is bounded, there exists $ \gamma >' 0$ such that $B \subset  \gamma U$. That is, $\lambda B \subset  U$, where $\lambda = \frac{1}{\gamma} > ' 0.$ Thus, $\{ \lambda B \; : \; \lambda >' 0 \}$ forms a neighbourhood base at $0$. This completes the proof.
\end{proof}
\end{section} 

\begin{remark} All the results in this paper can also be developed for topological hyperbolic modules on similar lines with slight adjustments.
\end{remark}

\bibliographystyle{amsplain}

\begin{thebibliography}{99}
%
\bibitem{YY} D. Alpay, M. E. Luna-Elizarraras, M. Shapiro and D. C. Struppa, {\em Basics of Functional Analysis with Bicomplex scalars, and Bicomplex Schur Analysis}, Springer Briefs in Mathematics, 2014.  
%
\bibitem{NB} N. Bourbaki, {\em Topological vector spaces}, Springer, Berlin, 1987.
%
 \bibitem{CS}  F. Colombo, I. Sabadini and D. C. Struppa, {\em Bicomplex holomorphic functional calculus}, Math. Nachr. \textbf{287}, No. 13 (2013), 1093-1105. 
%
\bibitem{H_6H_6}  F. Colombo, I. Sabadin, D. C. Struppa,  A. Vajiac and M. B. Vajiac, {\em Singularities of functions of one and several bicomplex variables}, Ark. Mat. \textbf{49}, (2011), 277-294.
%
\bibitem{JJ} J. B. Conway, {\em A Course in Functional Analysis}, 2nd Edition, Springer, Berlin, 1990.
%
\bibitem{ff} H. De Bie, D. C. Struppa, A. Vajiac and M. B. Vajiac, {\em The Cauchy-Kowalewski product for bicomplex holomorphic functions}, Math. Nachr. \textbf{285}, No. 10 (2012), 1230-1242.
%
\bibitem{DS} N. Dunford and J. T. Schwartz, {\em Linear Operators: Part I}, Wiley, New York, London, 1958.
%
 \bibitem{G_1G_1} R. Gervais Lavoie, L. Marchildon and D. Rochon, {\em Finite-dimensional bicomplex Hilbert spaces}, Adv. Appl. Clifford Algebr. \textbf{21}, N0. 3 (2011), 561-581.
%
\bibitem{GG} R. Gervais Lavoie, L. Marchildon and D. Rochon, {\em Infinite-dimensional bicomplex Hilbert spaces}, Ann. Funct. Anal. \textbf{1}, No. 2 (2010), 75-91.
%
\bibitem{LL} R. Kumar, R. Kumar and D. Rochon, {\em The fundamental theorems in the framework of bicomplex topological modules}, (2011), arXiv:1109.3424v1.
%
\bibitem{KS} R. Kumar and K. Singh, {\em Bicomplex linear operators on bicomplex Hilbert spaces and Littlewood's subordination theorem}, Adv. Appl. Clifford Algebr., DOI: 10.1007/s00006-015-0531-3,  (2015).  
%
\bibitem{RK} R. Kumar, K. Singh, H. Saini and S. Kumar {\em Bicomplex weighted Hardy  spaces and bicomplex C$^{*}$-algebras},  Adv. Appl. Clifford Algebr., DOI: 10.1007/s00006-015-0572-7, (2015).  
%
\bibitem{LA} R. Larsen, {\em Functional Analysis : An Introduction}, Marcel Dekker, New York, 1973.
%
\bibitem{Hahn} M. E. Luna-Elizarraras, C. O. Perez-Regalado and M. Shapiro, {\em On linear functionals and Hahn-Banach theorems for hyperbolic and bicomplex modules}, Adv. Appl. Clifford Algebr. \textbf{24}, (2014), 1105-1129.    
%
\bibitem{Luna} M. E. Luna-Elizarraras, C. O. Perez-Regalado and M. Shapiro, {\em On the bicomplex Gleason-Kahane Zelazko Theorem}, Complex Anal. Oper. Theory, DOI: 10.1007/s11785-015-0455-x, (2015).
%
\bibitem{MM} M. E. Luna-Elizarraras, M. Shapiro and D. C. Struppa, {\em On Clifford analysis for holomorphic mappings}, Adv. Geom. \textbf{14}, No. 3 (2014), 413-426.
%
\bibitem{M_1M_1} M. E. Luna-Elizarraras, M. Shapiro, D. C. Struppa and A. Vajiac, {\em Bicomplex numbers and their elementary functions}, Cubo \text{14}, No. 2 (2012), 61-80.
%
\bibitem{ZZ} M. E. Luna-Elizarraras, M. Shapiro, D. C. Struppa and A. Vajiac, {\em Complex Laplacian and derivatives of bicomplex functions}, Complex Anal. Oper. Theory \textbf{7} (2013), 1675-1711.  
%
\bibitem{NE} L. Narici, E. Beckenstein, {\em Topological Vector Spaces}, Marcel Dekker, New York 1985. 
%
\bibitem{KK} G. B. Price, {\em An Introduction to Multicomplex Spaces and Functions}, 3rd Edition, Marcel Dekker, New York, 1991.
%
\bibitem{Z_1Z_1} J. D. Riley, {\em Contributions to the theory of functions of a bicomplex variable}, Tohoku Math J(2) \textbf{5}, No.2 (1953), 132-165.
%
\bibitem{RR} D. Rochon and M. Shapiro, {\em On algebraic properties of bicomplex and hyperbolic numbers}, Anal. Univ. Oradea, Fasc. Math. \textbf{11} (2004), 71-110.
%
\bibitem{XX} D. Rochon and S. Tremblay, {\em Bicomplex Quantum Mechanics II: The Hilbert Space}, Advances in Applied Clifford Algebras, \textbf{16} No. 2 (2006), 135-157.
 %
\bibitem{HH} H. H. Schaefer, {\em Topological Vector Spaces}, Springer, Berlin, 1971.
%
\bibitem{rudin} W. Rudin, {\em Functional analysis}, 2nd Edition, McGraw Hill, New York, 1991.
%
\bibitem{YO} K. Yoshida, {\em Functional Analysis}, Springer-Verlag, Berlin, 1968. 
%


\end{thebibliography}

\noindent Romesh Kumar, \textit{Department of Mathematics, University of Jammu, Jammu, J\&K - 180 006, India.}\\
E-mail :\textit{ romesh\_jammu@yahoo.com}\\

\noindent Heera Saini, \textit{Department of Mathematics, University of Jammu, Jammu,  J\&K - 180 006, India.}\\
E-mail :\textit{ heerasainihs@gmail.com}\\

\end{document}